\documentclass[12pt,reqno]{amsart}
\usepackage{graphicx}
\usepackage{amssymb,amsmath}
\usepackage{amsthm}
\usepackage{color,graphicx}
\usepackage{hyperref}
\usepackage{color}
\usepackage{mathabx}

\newcommand{\be}{\begin{equation}}

\newcommand{\ee}{\end{equation}}

\numberwithin{equation}{section}
\numberwithin{figure}{section}

\newtheorem{theorem}{Theorem}[section]

\newtheorem{lemma}[theorem]{Lemma}

\begin{document}
\vglue-1cm \hskip1cm
\title[Monotonicity of the period map]{Monotonicity of the period map for the equation $-\varphi''+\varphi-\varphi^{k}=0$.}

\begin{center}


\keywords{Periodicity of the period map, Floquet theory, Periodic solutions.\\
\indent $^*$Corresponding author}

\maketitle

{\bf Giovana Alves }

{Centro de Ci\^encias Exatas Naturais e Tecnol\'ogicas\\
Universidade Estadual da Regi\~ao Tocantina do Maranh\~ao\\
	Imperatriz, Maranh\~ao, 65900-000, Brazil.}\\
{ a$\_$giovanaalves@yahoo.com.br}

{\bf F\'abio Natali$^*$}

{Departamento de Matem\'atica\\
Universidade Estadual de Maring\'a\\
Maring\'a, Paran\'a, CEP 87020-900, Brazil.}\\
{ fmanatali@uem.br}

\vspace{3mm}

\end{center}

\begin{abstract}
In this paper, we establish the monotonicity of the period map in terms of the energy levels for certain periodic solutions of the equation $-\varphi''+\varphi-\varphi^{k}=0$, where $k>1$ is a real number. We present a new approach to demonstrate this property, utilizing spectral information of the corresponding linearized operator around the periodic solution and tools related to Floquet theory.
\end{abstract}

\section{Introduction}
Let $\varphi:\mathbb{R}\rightarrow\mathbb{R}$ be a periodic solution of the equation
\be
-\varphi''+  \varphi -\varphi^k = 0, \label{ode1}
\ee
where $k>1$ is a real number and $\varphi=\varphi(x)$ is a function depending on $x\in\mathbb{R}$. We easily verify that , the function defined by
\begin{equation}\label{energyODE1}
	\mathcal{E}(\varphi, \xi) =  \frac{\xi^2}{2} -\frac{\varphi^2}{2}+\frac{\varphi^{k+1}}{k+1},
\end{equation}
is a first integral for the two dimensional system associated to equation \eqref{ode1}, where $\xi=\varphi'$. In other words, this means that the pair $(\varphi,\xi)$ satisfies $\mathcal{E}(\varphi,\xi)=B$ for convenient values $B\in\mathbb{R}$. When $k>1$ and $\frac{1-k}{2(k+1)}<B<0$, we obtain periodic orbits that spinning around the center point $(1,0)$. 
When $k>1$ is particular an odd number, we also obtain periodic solutions but for all $B>0$. In this case, the periodic solutions change their sign, and all initial conditions associated with the solution are located outside the separatrix curve. (a curve that converges to the saddle point $(0,0)$ in the associated phase portrait).\\
\indent The period map $L=L(B)$ of the periodic solution $\varphi$ depends on $B$ and it can be expressed by
\be
L=\displaystyle 2\int_{b_1}^{b_2}\frac{dh}{\sqrt{-\frac{2h^{k+1}}{k+1}+ h^2+2B}},
\label{persol1}\ee
where $b_1=\displaystyle\min_{x\in [0,L]}\varphi(x)$ and $b_2=\displaystyle\max_{x\in [0,L]}\varphi(x)$. This expression is not very handy to obtain the smoothness of the period function $L$ in terms of $B$ since $b_1$ and $b_2$ are zeroes of the function $G(\varphi)=-\frac{2\varphi^{k+2}}{k+2}+ \varphi^2+2B$. However, from standard theory of ODE, the solution $\varphi$ of $(\ref{ode1})$ depends smoothly on the initial conditions $\varphi(t_0)=\varphi_0$ and $\varphi'(t_0)=\varphi_1$, $t_0\in\mathbb{R}$ and in particular, $\varphi$ and $L$ also depend smoothly on $B$ in a convenient set of parameters (see \cite[Chapter 1, Theorem 3.3]{jack}).\\
\indent With the smoothness of the period map in hands, it is important to know the behaviour of this function in terms of the energy levels $B$ or even in terms of the initial data. In fact, questions related to the behaviour of the period function have attracted a considerable attention in the last decades and the study of the monotonicity of the period map is important for several reasons. For instance, if the period map is monotonic, solutions with large periods correspond to initial conditions that are far away from the critical points. On the other hand, if the period map is not monotonic, then the periodic solutions may exhibit bifurcations points in the associated phase portrait and the uniqueness of solutions may be not valid.\\
\indent Next, we present some contributors. Loud \cite{loud} and Urabe \cite{urabe} have proved the isochronicity\footnote{Just to make clear the comprehension of the reader, a center point is said to be \textit{isochronous} when all periodic solutions  turning around the center point have the same period} of the period function for specific families of planar vector fields. Using tools of ODE, Schaaf \cite{schaaf} has established a general set of sufficient conditions for the monotonicity of the period function for a class of Hamiltonian systems of the form
\be\label{hamiltgen}
\left\{\begin{array}{llll}\varphi'=-f(\psi),\\
\psi'=g(\varphi).\end{array}
\right.
\ee
By assuming suitable assumptions on the smooth functions $f$ and $g$, it is possible to prove that the period function is strictly increasing for periodic solutions  turning around the center point $(0,0)$ (see also Rothe \cite{rothe2} for an extension of the results in \cite{schaaf}).\\
\indent  Cima, Gasull and Ma\~nosas \cite{cima} (see also Coppel and Gavrilov \cite{coppel}) studied the same system in $(\ref{hamiltgen})$ and they characterized the limiting behaviour of the period $L$ at infinity when the origin is a global center. In adddition, they apply this result to prove that there are no nonlinear polynomial isochronous centers in this family, showing the monotonicity of the period in terms of the energy levels $B$. Rothe \cite{rothe1} and Waldvogel \cite{waldvogel}  showed that all Lotka-Volterra systems of the form
$$\varphi’=\varphi(a-b\psi), \quad \psi’=-\psi(c-d\varphi),$$
for $a$ and $b$ in a suitable set of conditions, have monotonic period functions. Chow and Wang \cite{chow} gave some characterization of the first and second derivatives for the period function in terms of $B$ for the equation
$$\varphi''+g(\varphi)=0,$$
where $g$ satisfies a suitable set of conditions. As an application, they determine the monotonicity of the period function when $g(s)=e^s-1$. Using a different approach, Chicone \cite{chicone2} also obtained part of the results in \cite{chow}  and gave a general condition for the monotonicity of period functions for a class of planar Hamiltonian systems. In addition, \cite{chicone2} was the inspiration for the study in Yagasaki \cite{yagasaki2} to prove the monotonicity of the period function for the equation $(\ref{ode1})$. In this work, the authors also considered $k>1$ a real number and the periodic solutions turning around the center point $(1,0)$. A revisited proof of 
\cite{yagasaki2} has been established by Benguria, Depassier and Loss in \cite{benguria}.\\
\indent In our approach, we use a different method to establish the monotonicity of the period map for positive solutions as determined in \cite{yagasaki2} and we realize that when $k>1$ is in particular an odd number, we can use the same method to determine the same behaviour for the periodic solutions of the equation $(\ref{ode1})$ that change their sign. As far as we know, this last question never been treated in the current literature and our intention is to give a new perspective to solve this kind of problems since our approach can be used in other models. \\
\indent Without further ado, we will describe our methods. Indeed, let us consider the  linearized operator around $\varphi$ given by 
\begin{equation}\label{operator1}
\mathcal{L}(y)=- y''+ y-k\varphi^{k-1}y,
\end{equation}
defined in $L_{per}^2([0,L])$ with domain $H_{per}^2([0,L])$. Deriving equation $(\ref{ode1})$ with respect to $x$, we obtain that $\varphi'$ satisfies $\mathcal{L}\varphi'=0$, so that $\varphi'$ is a periodic element which solves the equation 
\be\label{fundset}
- y''+ y-k\varphi^{k-1}y=0.
\ee Since $(\ref{fundset})$ is a second order linear equation, there exists another solution $\bar{y}$  for the equation $\mathcal{L}\bar{y}=0$ and 
$\{\varphi',\bar{y}\}$ is a fundamental set of solutions for the equation $(\ref{fundset})$. Using tools of the Floquet theory (see \cite{magnus}), we see that function $\bar{y}$ can be periodic or not and it is related with $\varphi'$ through the equality (see \cite{magnus})
\be\label{theta2}
\bar{y}(x+L)=\bar{y}(x)+\theta\varphi'(x),
\ee
for all $x\in \mathbb{R}$. Constant $\theta$ is a real parameter and it determines if $\bar{y}$ is periodic or not in the sense that $\bar{y}$ is periodic if and only if $\theta=0$.\\
\indent Our strategy to prove the monotonicity of $L$ in terms of $B$ is then to show that $\bar{y}$ is always a non-periodic function by proving that the kernel of $\mathcal{L}$ in $(\ref{operator})$ in $(\ref{operator1})$ defined in $L_{per}^2([0,L])$ with domain $H_{per}^2([0,L])$ is always simple  (see Lemma $\ref{teospec}$ and Lemma $\ref{condzero}$). After that, we need to use the smoothness of the period function with respect to $B$ to prove that $\frac{\partial L}{\partial B}=-\theta$ (see Lemma $\ref{teo1}$). Thus, since $\bar{y}$ is always non-periodic, we can conclude that $\theta \neq 0$, and the monotonicity of the period map is determined. Our main result can be summarized as follows:
\begin{theorem}\label{mainT}
Let $L=L(B)$ be the period function associated to the periodic solution $\varphi$ of the equation $(\ref{ode1})$.\\
i) If $k>1$ is odd and $\varphi$ is a solution that changes its sign, then $\frac{\partial L}{\partial B}<0$.\\
ii) If $k>1$ and $\varphi$ is a positive non-constant periodic solution, then $\frac{\partial L}{\partial B}>0$.
\end{theorem}
\indent Our paper is organized as follows. In Section 2 we present a brief explanation concerning the existence of positive and solutions that change their sign. Section 3 is devoted to prove Theorem $\ref{mainT}$.\\

\textbf{Notation.} Here, we introduce the basic notation concerning the periodic Sobolev spaces that will be useful in our paper. For a more complete introduction to these spaces we refer the reader to \cite{Iorio}. In fact, by $L^2_{per}([0,L])$, we denote the space of all square integrable functions which are $L$-periodic. For $m\geq0$ integer, the Sobolev space
$H^m_{per}([0,L])$
is the set of all periodic functions such that
$$
\|f\|_{H^m_{per}}^2:= ||f||_{L_{per}^2}^2+\sum_{j=1}^m||f^{(j)}||_{L_{per}^2}^2 <\infty,
$$
where $f^{(j)}$, $j=1,\cdots,m$, indicates the $j$-th derivative of $f$ in the sense of periodic distributions (see \cite{Iorio} for further details). The space $H^m_{per}([0,L])$ is a  Hilbert space with natural inner product denoted by $(\cdot, \cdot)_{H_{per}^m}$. When $m=0$, the space $H^m_{per}([0,L])$ is isometrically isomorphic to the space  $L^2_{per}([0,L])$. The norm and inner product in $L^2_{per}([0,L])$ will be denoted by $\|\cdot \|_{L_{per}^2}$ and $(\cdot, \cdot)_{L_{per}^2}$, respectively.

\section{Existence of periodic solutions via planar analysis.}
Our purpose in this section is to present some facts concerning the existence of
periodic solutions for the nonlinear ODE given by
\be
-\varphi''+  \varphi -\varphi^k = 0, \label{ode}
\ee
where $k>1$ is a real number. 

Equation (\ref{ode}) is
conservative with 
\begin{equation}\label{energyODE}
\mathcal{E}(\varphi, \xi) =  \frac{\xi^2}{2} -\frac{\varphi^2}{2}+\frac{\varphi^{k+1}}{k+1},
\end{equation}
as a first integral,  where $\xi=\varphi'$, and thus its solutions are contained on the level curves of the energy.\\
\indent According to the classical ODE theory (see \cite{chicone}, \cite{jack} and \cite{natali1} for further details), $\varphi$ is a periodic solution of the equation $(\ref{ode})$ if and only if $(\varphi,\varphi')$ is a periodic orbit of the planar differential system
\begin{equation}\label{planarODE}
\left\{\begin{array}{lllll}
\varphi'=\xi,\\\\
\xi'=\varphi-\varphi^k.
\end{array}\right.
\end{equation}
\indent The periodic orbits for the equation $(\ref{planarODE})$ can be determined by considering the energy levels of the function $\mathcal{E}$ defined in $(\ref{energyODE})$. This means that the pair $(\varphi,\xi)$ satisfies the equation $\mathcal{E}(\varphi,\xi)=B$. If $k>1$ and $B\in \left(\frac{1-k}{2(k+1)},0\right)$, we obtain periodic orbits that turn round at the equilibrium points $(1,0)$ and these orbits are called \textit{positive} since they are associated with positive  periodic solutions of the equation $(\ref{ode})$. In our specific case, $(\ref{planarODE})$ admits at least two critical points, being a saddle point at the origin and a center point at $(1 ,0)$. According to the standard ODE theory, the periodic orbits emanate from the center points to the separatrix curve which is represented by a smooth solution $\widetilde{\varphi}:\mathbb{R}\rightarrow\mathbb{R}$ of $(\ref{ode})$ satisfying $\lim_{x\rightarrow \pm \infty}\widetilde{\varphi}^{(n)}(x)=0$ for all $n\in\mathbb{\mathbb{N}}$. 
In particular, when $k$ is odd, we see that the presence of two symmetric center points $(\pm 1,0)$ allows to conclude that the periodic orbits which turn around these points can be positive and negative\footnote{The terminology \textit{positive (negative) periodic orbit} means that the associated periodic solution for the equation $(\ref{ode})$ is positive (negative).}. For every orbit with initial condition outside the separatrix, we obtain periodic solutions for the equation $(\ref{ode})$ that change their sign \footnote{Here, we can also define an appropriate terminology for this kind of periodic solution as \textit{periodic orbit that changes its sign.}}. Indeed, if $B>0$ we also have periodic orbits and the corresponding periodic solutions  $\varphi$ with the zero mean property, that is, periodic solutions satisfying  $\int_0^L\varphi(x)dx=0$.
\begin{figure}[h]\label{fig1}
	\centering
	\begin{minipage}[h]{0.5\linewidth}
		\includegraphics[scale=0.21]{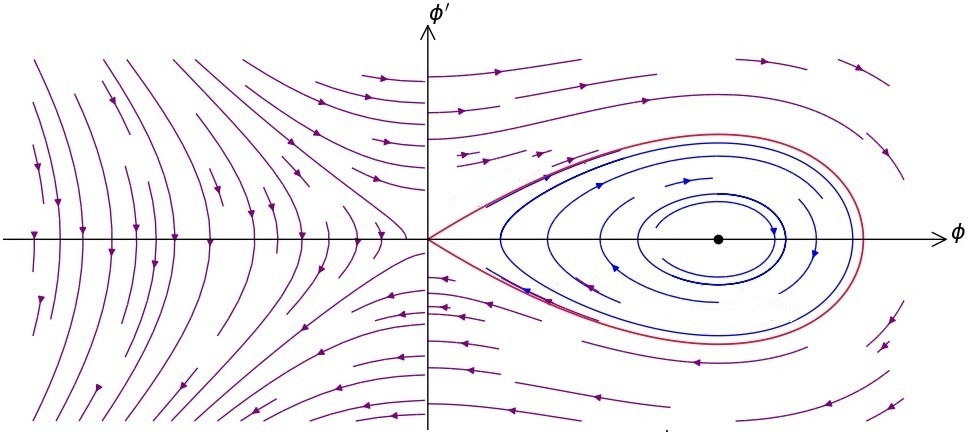}
	\end{minipage}\hfill
	\begin{minipage}[h]{0.5\linewidth}
		\includegraphics[scale=0.21]{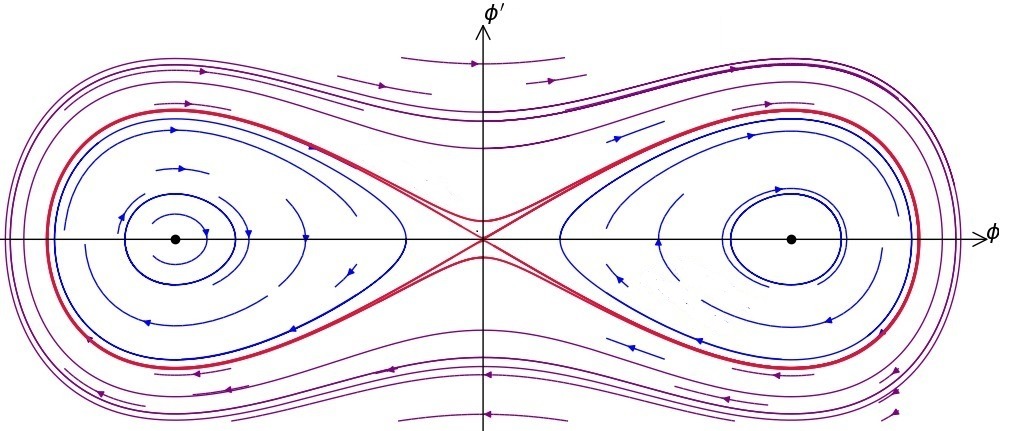}
	\end{minipage}
\caption{Periodic orbits for the equation $(\ref{ode})$. Left: positive and periodic orbits for the case $k>1$ that spinning around the center point $(1,0)$. Right: Positive and negative periodic orbits for the case $k>1$ and odd that spinning around the center points $(1,0)$ and $(-1,0)$, respectively. Right: periodic orbits that change their sign when the initial conditions are located outside the the separatrix.}
\end{figure}

Independently of the type of periodic solutions on which we are working, the period $L=L(B)$ of the solution $\varphi$ can be expressed (formally) as follows:

\be
L=\displaystyle 2\int_{b_1}^{b_2}\frac{dh}{\sqrt{-\frac{2h^{k+1}}{k+1}+ h^2+2B}},
\label{persol}\ee
where $b_1=\displaystyle\min_{x\in [0,L]}\varphi(x)$ and $b_2=\displaystyle\max_{x\in [0,L]}\varphi(x)$.\\
\indent On the other hand, the energy levels of the first integral $\mathcal{E}$ in $(\ref{energyODE})$ parametrize the unbounded set of periodic orbits $\{\Gamma_B\}_B$ which emanate from the separatrix curve. Thus, we can conclude that the set of smooth periodic solutions of $(\ref{ode})$ can be expressed by a smooth family $\varphi=\varphi_{B}$ which is parametrized by the value $B$. 
Moreover, for the case $B\in \left(\frac{1-k}{2(k+1)},0\right)$, we see that if $B\rightarrow \frac{1-k}{2(k+1)}$ thus $L\rightarrow  2\pi$, where $2\pi$ is the period of the equilibrium solution $(1,0)$ of the system $(\ref{planarODE})$ (see \cite[Section 2]{natali1}). On the other hand, if $B\rightarrow0$ we have $L\rightarrow +\infty$ and both convergences suggest that the period map is increasing in the interval $\left(\frac{1-k}{2(k+1)},0\right)$. On the other hand, when $B\in (0,+\infty)$, we see that if $B\rightarrow 0$, we have $L\rightarrow +\infty$, and if $B\rightarrow+\infty$, we obtain $L\rightarrow 0$. In this case, the suggestion is that the period map is decreasing.

\section{The monotonicity of the period map - proof of Theorem $\ref{mainT}$.}

\indent We need to recall some basic facts concerning Floquet's theory (see \cite{est} and \cite{magnus}). Let $Q$ be a smooth $L$-periodic function.  Consider $\mathcal{P}$ the Hill operator defined in $L_{per}^2([0,L])$, with domain $D(\mathcal{P})=H_{per}^2([0,L])$, given by
\begin{equation}\label{operatorP}
\mathcal{P}=-\partial_x^2+Q(x).
\end{equation}
The spectrum of $\mathcal{P}$  is formed by an unbounded sequence of
real eigenvalues
\be\label{seqeigen}
\lambda_0 < \lambda_1 \leq \lambda_2 \leq \lambda_3 \leq \lambda_4\leq
\cdots\; \leq \lambda_{2n-1} \leq \lambda_{2n}\; \cdots,
\ee
where equality means that $\lambda_{2n-1} = \lambda_{2n}$  is a
double eigenvalue. Moreover, according with the Oscillation Theorem  (see \cite{magnus}), the spectrum  is characterized by the number of zeros
of the eigenfunctions as: if $p$ is an eigenfunction associated to either $\lambda_{2n-1}$ or $\lambda_{2n}$, then $p$  has exactly
$2n$ zeros in the half-open
interval $[0, L)$. In particular, the even eigenfunction associated to the first eigenvalue $\lambda_0$ has no zeros in $[0, L]$.

Let $p$ be a nontrivial $L$-periodic solution of the equation
\begin{equation}\label{zeqL}
\mathcal{P}(p)=-p''+Q(x)p=0.
\end{equation}
Consider $\bar{y}$ the other solution of \eqref{zeqL} linearly independent of $p$.  There exists a constant $\theta$ (depending on $\bar{y}$ and $p$) such that (see \cite[page 5]{magnus})
\be\label{theta1}
\bar{y}(x+L)=\bar{y}(x)+\theta p(x).
\ee
Consequently, $\theta=0$ is a necessary and sufficient condition to all solutions of \eqref{zeqL} to be $L$-periodic. This criterion is very useful to establish if the kernel of $\mathcal{P}$ is one-dimensional or not.\\
\indent Next, let $\varphi=\varphi_{B}$ be any periodic solution of \eqref{ode}. Consider $\mathcal{L}=\mathcal{L}_{B}:H_{per}^2([0,L])\subset L_{per}^2([0,L])\rightarrow L_{per}^2([0,L])$ the linearized operator arising from the linearization of \eqref{ode} at $\varphi$, that is,
\begin{equation}\label{operator}
\mathcal{L}(y)=- y''+ y-k\varphi^{k-1}y.
\end{equation}
Clearly $\mathcal{P}=\mathcal{L}$, when $Q(x)y(x)= y(x)-k \varphi^{k-1}(x)y(x)$. By taking the derivative with respect to $x$ is \eqref{ode}, we have that $p:=\varphi'$ belongs to the kernel of the operator  $\mathcal{L}$. In addition, from our construction, $\varphi'$ has exactly
two zeros in the half-open interval $[0,L)$ (see Figure 2.1), which implies that zero is the second or the third eigenvalue of  $\mathcal{L}$ (see Oscillation's Theorem in \cite{magnus}) 

Next, let $\bar{y}=\bar{y}_{B}$ be the unique even solution of the initial-value problem
\be \left\{
\begin{array}{l}
  -\bar{y}'' +\bar{y}-k\varphi^{k-1}\bar{y} = 0 ,\\
 \bar{y}(0) =-\frac{1}{\varphi''(0)}, \\
\bar{y}'(0)=  0.
 \end{array} \right.
\label{y}
\ee
Since $\varphi'$ is an $L$-periodic solution for the equation in \eqref{y} and the corresponding Wronskian of $\varphi'$ and $\bar{y}$ is 1, there is a constant $\theta=\theta_{\bar{y}}$ such that
\begin{equation}\label{y1}
\bar{y}(x+L)=\bar{y}(x)+\theta \varphi'(x).
\end{equation}
By taking the derivative in this last expression and evaluating at $x=0$, we obtain
\be \label{theta}
\theta= \frac{\bar{y}'(L)}{\varphi''(0)}.
\ee

\indent The next result gives that it is possible to decide the exact position of the zero eigenvalue by knowing the precise sign of $\theta$ in $(\ref{theta1})$.

\begin{lemma}\label{specprop}
	i) If $\theta>0$, operator $\mathcal{L}$, defined in $L_{per}^2([0,L])$, with domain $H_{per}^2([0,L])$, has exactly two negative eigenvalue which are simple, a simple eigenvalue at zero and the rest of the spectrum is positive and bounded away from zero.\\
	ii) If $\theta<0$, the same operator $\mathcal{L}$ has exactly one negative eigenvalue which is simple, a simple eigenvalue at zero and the rest of the spectrum is positive and bounded away from zero.
\end{lemma}
\begin{proof}
	See Theorem 3.1 in \cite{neves}.
\end{proof}






\indent Now we can give a relation between $\frac{\partial L}{\partial B}$ and $\theta$.

\begin{lemma} \label{teo1}
We have that $\frac{\partial L}{\partial B}=-\theta$, where $\theta$ is the constant in \eqref{theta}.
\end{lemma}
\begin{proof}
Consider $\bar{y}$ and $\varphi'$ as above. Since $\varphi$ is even and periodic one has $\varphi'(0)=\varphi'(L)=0$. Thus, the smoothness of $\varphi$ in terms of the parameter $B$ enables us to take the derivative of $\varphi'(L)=0$ with respect to $B$ to obtain 
\begin{equation}\label{eq12343}
\varphi''(L)\frac{\partial L}{\partial B}+\frac{\partial \varphi'(L)}{\partial B}=0.\end{equation}
\indent Next, we turn back to equation $(\ref{ode})$ and multiply it by $\varphi'$ to deduce, after integration, the quadrature form
\begin{equation}\label{eq12344}
\frac{ \varphi'^2(x)}{2}-\frac{\varphi(x)^2}{2}+\frac{\varphi(x)^{k+1}}{k+1}-B=0,\ \ \ \ \mbox{for all}\ x\in [0,L].
\end{equation}
Deriving equation $(\ref{eq12344})$ with respect to $B$ and taking $x=0$ in the final result, we obtain from $(\ref{ode})$ that 
$\frac{\partial \varphi(0)}{\partial B}=\frac{1}{ \varphi''(0)}$. In addition, since $\varphi$ is even one has that $\frac{\partial \varphi}{\partial B}$ is also even and thus $\frac{\partial \varphi'(0)}{\partial B}=0$. On the other hand, deriving equation $(\ref{ode})$ with respect to $B$, we obtain that $\frac{\partial\varphi}{\partial B}$ satisfies the initial-value problem
\begin{equation}\label{varphiB}
\left\{\begin{array}{lllll}\displaystyle-\frac{\partial\varphi}{\partial B}''+\frac{\partial \varphi}{\partial B}-k\varphi^{k-1}\frac{\partial \varphi}{\partial B}=0,\\\\
\displaystyle\frac{\partial \varphi(0)}{\partial B}=\frac{1}{ \varphi''(0)},\\\\
\displaystyle\frac{\partial \varphi'(0)}{\partial B}=0.
\end{array}\right.	
\end{equation}
The existence and uniqueness theorem for ordinary differential equations applied to the problem $(\ref{y})$ enables us to deduce that $\bar{y}=-\frac{\partial \varphi}{\partial B}$. Therefore, we can combine $(\ref{theta})$ with $(\ref{eq12343})$ to obtain that $\frac{\partial L}{\partial B}=-\theta$. 
\end{proof}

\indent Next lemma is basic for our purposes. 

\begin{lemma}\label{1range}
	Let $\mathcal{L}$ be the linearized operator defined in $(\ref{operator})$. Thus, $1\in \rm{range}(\mathcal{L})$. 	
\end{lemma}
\begin{proof}
	First, we see that (\ref{ode}) is invariant under translations. This means that if $\varphi$ is a solution of $(\ref{ode})$, thus $\psi_r=\varphi(\cdot-r)$ is also a solution for all $r\in\mathbb{R}$. In particular, if $r=L/4$ we obtain that $\psi:=\psi_{L/4}=\varphi(\cdot-L/4)$ is also a periodic solution for the equation $(\ref{ode})$. Therefore, $\varphi'$ is an odd element of $\ker(\mathcal{L})$, so that $\psi'$ results to be even and it is an element of the kernel of $\widetilde{\mathcal{L}}=-\partial_x^2+1-k\psi^{k-1}$ ($\widetilde{\mathcal{L}}$ is the operator obtained from $\mathcal{L}$ by considering the linearization at $\psi$ instead of $\varphi$). Consider then, $\chi=\bar{y}(\cdot-L/4)$ the corresponding element of the formal equation $\mathcal{L}(\bar{y})=0$ associated to the translation solution $\psi$. Since $\bar{y}$ is even, $\chi=\bar{y}(\cdot-L/4)$ and $\bar{y}(\cdot+L/4)$ result to be odd. Considering the change of variables $x=s+L/2$, we obtain that
	$$\int_0^L\chi(x)dx
	 \int_0^L\bar{y}(x-L/4)dx=\int_{-L/2}^{L/2}\bar{y}(s+L/4)ds=0,$$ 
	that is, $\chi$ has the zero mean property.\\
	\indent Let us define the following function:
	\begin{equation}\label{varform}
		h(x)=\left(\int_0^x \chi(s)ds\right) \psi'(x)-\left(\int_0^x\psi'(s)ds\right)\chi(x).
		\end{equation}	
Using the method of variation of parameters, we see that $h$ is a particular solution (not necessarily periodic) of the equation 
\be\label{gensol}-u''+u-k\psi^{k-1}u=1,\ee 
in the sense that $Y=c_1\psi'+c_2\chi +h$ is a general solution associated with the equation $(\ref{gensol})$, where $c_1$ and $c_2$ are real constants. We need to prove that $h$ is periodic. In fact, it is clear $h(0)=0$ and since $\int_0^L\chi(x)dx=0$, one has from $(\ref{varform})$ and the fact that $\psi'$ is periodic that $h(L)=0$. Deriving $(\ref{varform})$ with respect to $x$, we obtain after some computations $h'(x)=\left(\int_0^x \chi(s)ds\right) \psi''(x)-\left(\int_0^x\psi'(s)ds\right)\chi'(x)$, so that $h'(0)=0$. Since $\psi$ is both odd and periodic, it follows that $\psi'$ and $\psi''$ are also periodic, with $\psi''$ being an odd function. We have
$$\begin{array}{lll}h'(L)&=&\displaystyle\left(\int_0^L \chi(s)ds\right) \psi''(L)-\left(\int_0^L\psi'(s)ds\right)\chi'(L)\\\\
&=&\displaystyle\left(\int_0^L \chi(s)ds\right) \psi''(0)-(\psi(L)-\psi(0))\chi'(L)=0\end{array}$$ so that, $h$ is periodic. Equality $(\ref{gensol})$ is then satisfied for the periodic function $h$, that is, $$-h''(x)+h(x)-k\psi(x)^{k-1}h(x)=1,$$ for all $x\in\mathbb{R}$. In particular, for $x=t+L/4$ and since $\varphi(t)=\psi(t+L/4)$, we have $$-h''(t+L/4)+h(t+L/4)-k\varphi(t)^{k-1}h(x+L/4)=1.$$ 
\indent Defining $\tilde{h}(t)=h(t+L/4)$, we obtain that $\tilde{h}$ is $L-$periodic and it satisfies the equation $-\tilde{h}''(t)+\tilde{h}(t)-k\varphi(t)^{k-1}\tilde{h}(t)=1$, so that $\mathcal{L}(\tilde{h})=1$ as requested in lemma.
\end{proof}

\begin{lemma}\label{phirange}
	We have that $\{\varphi^{k-1},\varphi^k\}\subset {\rm range}(\mathcal{L})$.
	\end{lemma}
\begin{proof}
	By $(\ref{ode})$, we obtain 
	$$\mathcal{L}(\varphi)=-\varphi''+\varphi-k\varphi^k=-\varphi''+\varphi-k\varphi^k+\varphi^k-\varphi^k=(1-k)\varphi^k.$$
	Thus, $\varphi^k\in {\rm range}(\mathcal{L})$.\\
	\indent To prove that $\varphi^{k-1}\in {\rm range}(\mathcal{L})$, we need to use the fact that $\mathcal{L}(1)=1-k\varphi^{k-1}$. By Lemma $\ref{1range}$ one has $1\in{\rm range}(\mathcal{L})$, so that $\varphi^{k-1}\in {\rm range}(\mathcal{L})$.
\end{proof}

\indent Next lemma establishes item i) of Theorem $\ref{mainT}$. In what follows $n(\mathcal{A})$ and $z(\mathcal{A})$ are respectively the number of negative eigenvalues (counting multiplicities) and the dimension of the kernel of a certain linear operator $\mathcal{A}$.

\begin{lemma}\label{condzero}
	Let $k>1$ be a positive odd integer. If $\varphi$ is the zero mean periodic solution of $(\ref{ode})$, thus $n(\mathcal{L})=2$ and $z(\mathcal{L})=1$. In particular, we have that $\frac{\partial L}{\partial B}<0$. 
\end{lemma}
\begin{proof}
	First, we see that $\varphi'$ is an odd eigenfunction of $\mathcal{L}$ associated to the eigenvalue $0$ having two zeroes in the interval $[0,L)$. From the Oscillation theorem (see \cite{magnus}), we obtain that $0$ needs to be the second or the third eigenvalue in the sequence of real numbers in $(\ref{seqeigen})$.\\
	\indent On the other hand, since $\varphi$ is even and $\int_0^L\varphi(x) dx=0$, one has $\psi=\varphi(\cdot-L/4)$ is odd and it has the zero mean property. Consider again $\widetilde{\mathcal{L}}=-\partial_x^2+1-k\psi^{k-1}$ the translated operator obtained from $\mathcal{L}$ by considering the linearization at $\psi$ instead of $\varphi$ and let $\widetilde{\mathcal{L}}_{odd}$ be the restriction of $\widetilde{\mathcal{L}}$ in the odd sector of $L_{per}^2([0,L])$. Notice that such restriction is possible because $\psi^{k-1}$ is even since $k-1$ is an even number. Since $k+1>2$ is also even, we have $(\widetilde{\mathcal{L}}_{odd}(\psi),\psi)_{L_{per}^2}=(1-k)\int_{0}^L\psi(x)^{k+1}dx<0$, so that by Courant's min-max characterization of eigenvalues, we obtain $n(\widetilde{\mathcal{L}}_{odd})\geq1$. The fact  $n(\widetilde{\mathcal{L}})=n(\widetilde{\mathcal{L}}_{odd})+n(\widetilde{\mathcal{L}}_{even})$ and  Krein-Rutman's Theorem enable us to conclude that the first eigenvalue of $\widetilde{\mathcal{L}}$ is simple and it is associated to a positive (negative) eigenfunction which needs to be even. Thus, we obtain since $0$ is the second or third eigenvalue of $\widetilde{\mathcal{L}}$ that $n(\widetilde{\mathcal{L}})=n(\mathcal{L})=2$ as requested.\\
	\indent We prove that $z(\mathcal{L})=1$. Indeed, since $n(\widetilde{\mathcal{L}}_{odd})=1$, we see that the corresponding eigenfunction $p$ associated to the first eigenvalue of $\widetilde{\mathcal{L}}_{odd}$ is odd and consequently, $q=p(\cdot-L/4)$ is an even function that changes its sign. Again, by Krein-Rutman's theorem we have that the first eigenfunction $\lambda_1$ of $\mathcal{L}$ is simple and it is associated to a positive (negative) even periodic function, so that $0$ can not be an eigenvalue associated to $\widetilde{\mathcal{L}}_{odd}$. Since $z(\widetilde{\mathcal{L}})=z(\widetilde{\mathcal{L}}_{odd})+z(\widetilde{\mathcal{L}}_{even})$, we obtain from the fact $\psi'$ is even that $z(\widetilde{\mathcal{L}})=z(\widetilde{\mathcal{L}}_{even})=1$. Therefore, using the translation transformation $f=g(\cdot-L/4)$, we obtain $z(\mathcal{L})=z(\mathcal{L}_{odd})=1$ as requested.
\end{proof}

\indent We prove the second part of Theorem $\ref{mainT}$. To do so, it is necessary to prove the following basic result.

\begin{lemma}\label{lemamin}
Let $L>0$ and $k>1$ be fixed. There exists $\varphi\in H_{per}^1([0,L])$ solution of the following constrained minimization problem
	\begin{equation}\label{minP1}
	\nu=\inf\left\{D(u);\ u\in H_{per}^1([0,L]),\ \int_0^{L} u(x)^{k+1}dx=1\right\},
\end{equation}
where $D: H_{per}^1([0,L])\rightarrow \mathbb{R}$	is the functional defined as
\begin{equation}\label{funcB1}
	D(u)=\frac{1}{2}\int_0^{L}u'(x)^2+u(x)^2dx.
\end{equation}
	In addition, $\varphi$ is smooth and it satisfies equation $(\ref{ode})$. Conversely, if $\varphi$ is a positive non-constant solution of $(\ref{ode})$, then $\varphi$ satisfies the minimization problem $(\ref{minP1})$.
	\end{lemma}
\begin{proof}
	Since $D$ is smooth and $D(u)\geq 0$ for all $u\in H_{per}^1([0,L])$, there exists $(u_n)_{n\in\mathbb{N}}$ a minimizing sequence associated to the problem $(\ref{minP1})$, that is, there exists $(u_n)_{n\in\mathbb{N}}\subset H_{per}^1([0,L])$ satisfying $\int_0^Lu_n(x)^{k+1}dx=1$ for all $n\in\mathbb{N}$, and 
	\begin{equation}D(u_n)\rightarrow \nu\label{conv1},
	\end{equation} 
as $n\rightarrow +\infty$. The convergence in $(\ref{conv1})$ allows us to conclude that $(u_n)_{n\in\mathbb{N}}$ is a bounded sequence in $H_{per}^1([0,L])$. Since $H_{per}^1([0,L])$ is a Hilbert space, there exists $\varphi\in H_{per}^1([0,L])$ such that 
	\begin{equation}u_n\rightharpoonup \varphi\label{conv2}.
\end{equation} 
On the other hand, the compact embedding $H_{per}^1([0,L])\hookrightarrow L_{per}^{k+1}([0,L])$ gives us	
\begin{equation}u_n\rightarrow \varphi\ \mbox{in}\ L_{per}^{k+1}([0,L]).\label{conv3}
\end{equation} 
\indent Consider the function $q:\mathbb{R}\rightarrow\mathbb{R}$ given by $q(s)=s^{k+1}$. The mean value theorem gives us $s^{k+1}-t^{k+1}=(k+1)w^{k}(s-t)$, where $w$ is point in the open interval $(t,s)$. Since $|w|\leq|s|+|t|$, we obtain \begin{equation}\label{est1}|u_n(x)^{k+1}-\varphi(x)^{k+1}|\leq 2^{k+1}(k+1)(|u_n(x)|^{k}+|\varphi(x)|^k)|u_n(x)-\varphi(x)|,
\end{equation} for all $x\in[0,L]$. The H\"older inequality and $(\ref{est1})$ allow us to conclude
\begin{equation}\label{conv4}\begin{array}{llll}
\displaystyle\left|\int_0^Lu_n(x)^{k+1}-\varphi(x)^{k+1}dx\right|&\leq&\displaystyle\int_0^L|u_n(x)^{k+1}-\varphi(x)^{k+1}|dx\\\\
&\leq& C_0\displaystyle||u_n-\varphi||_{L_{per}^{k+1}}(||u_n||_{L_{per}^{k+1}}^k+||\varphi||_{L_{per}^{k+1}}^k),
\end{array}\end{equation}
where $C_0=2^{k+1}(k+1)$. By the convergence in $(\ref{conv3})$, we obtain from $(\ref{conv4})$ that $\int_0^L\varphi(x)^{k+1}dx=1$ and thus, $\nu\leq D(\varphi)$. On the other hand, the weak lower semi-continuity of $D$ and the convergence in $(\ref{conv2})$ give us that $D(\varphi)\leq \lim\inf_{n\rightarrow +\infty}D(u_n)=\nu$. Thus $D(\varphi)=\nu$ and the infimum is attained at $\varphi$.\\
\indent An application of the Lagrange multiplier theorem guarantees the existence of $C_1$ such that
	\begin{equation}\label{lagrange}
		-\varphi''+\varphi=C_1\varphi^{k+1}.
	\end{equation}
	 The function $\varphi$ is nontrivial because $\int_0^L\varphi(x)^{k+1}dx=1$. Using a standard rescaling argument, we can deduce that the Lagrange multiplier $C_1$ can be chosen as $1$, and thus we deduce that $\varphi$ solves equation $(\ref{ode})$. In addition, a bootstrap argument applied to the equality in $(\ref{ode})$ also gives that $\varphi$ is smooth. Next, multiplying equation $(\ref{ode})$ by $\varphi$ and integrating the result over $[0,L]$, we conclude that $\nu=\frac{1}{2}$.\\
	\indent Conversely, suppose that $\varphi$ is a positive non-constant periodic solution of $(\ref{ode})$. According to the arguments presented in Section 2, we can deduce that the period of $\varphi$ satisfies $L>2\pi$. Additionally, since $\varphi>0$ and non-constant, we can assume, without loss of generality, that $\int_0^L \varphi(x)^{k+1}dx=1$. Multiplying the equation $(\ref{ode})$ by $\varphi$ and integrating the result over $[0,L]$, we obtain $D(\varphi)=\frac{1}{2}$. Since every solution $\phi$ of the minimization problem $(\ref{minP1})$ satisfies $D(\phi)=\nu=\frac{1}{2}$, we obtain that the positive non-constant periodic solution $\varphi$ also satisfies the minimization $(\ref{minP1})$ as requested in lemma.
\end{proof}
\indent The proof of the second part of Theorem $\ref{mainT}$ is now presented.
\begin{lemma}\label{teospec}
	Let $\varphi$ be a positive non-constant periodic solution associated with the equation $(\ref{ode})$. We have that $n(\mathcal{L})=z(\mathcal{L})=1$ and in particular $\frac{\partial L}{\partial B}>0$.
\end{lemma}
\begin{proof}
	Let $k>1$ be fixed. For $L>2\pi$, we have by Lemma $\ref{lemamin}$ that the positive non-constant solution $\varphi$ of the equation $(\ref{ode})$ satisfies the minimization problem $(\ref{minP1})$.\\
	\indent Next, thanks to the Lemma $\ref{lemamin}$, we obtain that $\varphi$ is a minimizer of the functional $G:H_{per}^1([0,L])\rightarrow \mathbb{R}$ given by
	$G(u)=D(u)-\frac{1}{k+1}\int_0^Lu(x)^{k+1}dx$ subject to the same constraint $\int_0^Lu(x)^{k+1}dx=1$. Since $\mathcal{L}=-\partial_x^2+1-k\varphi^{k-1}$ is the Hessian operator for $G(u)$ at the point $\varphi$, we obtain that $n(\mathcal{L})\leq1$. On the other hand, solution $\varphi$ is positive and thus $$(\mathcal{L}(\varphi),\varphi)_{L_{per}^2}=(1-k)\int_0^L\varphi(x)^{k+1}dx<0.$$ We deduce by Courant's min-max characterization of eigenvalues that $n(\mathcal{L})\geq1$ and both results allow to conclude that $n(\mathcal{L})=1$.  Let us assume that $z(\mathcal{L})=2$ for some $B_0\in\left(\frac{1-k}{2(k+1)},0\right)$. By Oscillation's Theorem in \cite{magnus}, it follows that the periodic solution $\bar{y}$ of the Cauchy problem $(\ref{y})$ has exactly two zeroes in the interval $[0,L)$ and by periodicity, also in the interval $\left[-L/2,L/2\right)$. Since $\varphi' \in \ker(\mathcal{L})$ is odd, we see that the periodic function $\bar{y} \in \ker(\mathcal{L})$ is even and it has exactly two symmetric zeroes  in the interval $\left[-L/2,L/2\right)$. Hence, there exists $x_0 \in \left(-L/2,L/2\right)$ such that $\bar{y}(\pm x_0)=0$. Without loss of generality, we can still suppose that
	\begin{equation}\label{positivityh}\bar{y}(x)>0, \; x \in (-x_0, x_0) \quad \text{and} \quad  \bar{y}(x)<0, \; x \in\left [-L/2,-x_0 \right) \cup \left(x_0,L/2\right).
	\end{equation}
	
	Furthermore, by Lemma $\ref{phirange}$ we have $\varphi^{k-1},\varphi^k\in\ker(\mathcal{L})^{\bot}={\rm range}(\mathcal{L})$, so that
	\begin{equation}\label{orthogonalityL1}
		(\bar{y},\varphi^{k-1})_{L^2_{per}}=0 \quad \text{and} \quad (\bar{y},\varphi^{k})_{L^2_{per}}=0.
	\end{equation}
	
	Since $\varphi>0$,  we obtain that
	$
	\varphi(x)^{k-1}(\varphi(x)-\varphi(x_0))
	$
	is positive over $(-x_0,x_0)$ and negative over $\left [-L/2,x_0 \right) \cup \left(x_0,L/2\right)$ and it has the same behaviour as $\bar{y}$ in \eqref{positivityh}. Thus, $(\varphi^{k-1}(\varphi-\varphi(x_0)), \bar{y})_{L^2_{per}} \neq 0$ which leads a contradiction with \eqref{orthogonalityL1}. Consequently, we have $\ker(\mathcal{L})=[\varphi']$.
\end{proof}

 \section*{Acknowledgments}
F. Natali is partially supported by CNPq/Brazil (grant 303907/2021-5) and CAPES MathAmSud (grant 88881.520205/2020-01).

 \section*{Data availability} Data sharing not applicable to this article as no datasets were generated or analysed during the current study.

\section*{Conflict of interest} The authors declare that they have no conflict
of interest.

\end{document}